\documentclass[11pt]{article}

\usepackage{amsmath,amsfonts,amssymb,amsthm}
\usepackage{enumerate}
\usepackage[hyphens]{url}
\usepackage{tcolorbox}
\usepackage{hyperref}
\usepackage{xspace}
\usepackage[margin=4cm]{geometry}
\usepackage{marvosym}

\theoremstyle{plain}
\newtheorem{theorem}{Theorem}

\newtheorem{proposition}[theorem]{Proposition}
\newtheorem{lemma}[theorem]{Lemma}

\theoremstyle{definition}
\newtheorem{definition}[theorem]{Definition}
\newtheorem{example}[theorem]{Example}

\newcommand\N{\mathbb{N}}

\newcommand\R{\mathbb{R}}

\newcommand{\DD}[2]{\frac{\partial #1}{\partial #2}}

\DeclareMathOperator{\im}{im} 
\DeclareMathOperator{\sign}{sign}
\DeclareMathOperator{\e}{e} 
\DeclareMathOperator{\diag}{diag}
\DeclareMathOperator{\rank}{rank}

\DeclareMathOperator{\conv}{conv}
\DeclareMathOperator{\zero}{zero}
\DeclareMathOperator{\supp}{supp}
\DeclareMathOperator{\interior}{int}
\DeclareMathOperator{\relint}{relint}

\DeclareMathOperator{\ver}{ver}

\DeclareMathOperator*{\argmax}{arg\,max}

\DeclareMathOperator{\id}{Id}
\newcommand{\ones}{1}
\newcommand{\trans}{\mathsf{T}}

\newcommand{\nP}{n_{\!P}}
\newcommand{\dP}{d_{\!P}}
\newcommand{\basT}{G}
\newcommand{\xx}{\xi}
\newcommand{\xv}{\xi}
\newcommand{\la}{\lambda}

\newcommand{\map}{F}
\newcommand{\face}{f}

\newcommand{\mono}{\green{monomial}\xspace}
\newcommand{\expo}{\green{exponential moment}\xspace}

\definecolor{darkgreen}{rgb}{0.0,0.7,0.0}
\newcommand\green[1]{{\textcolor{black}{#1}}}

\newcommand\blfootnote[1]{%
  \begin{NoHyper}
  \renewcommand\thefootnote{}\footnote{#1}%
  \addtocounter{footnote}{-1}%
  \end{NoHyper}
}

\usepackage{algorithm, algpseudocode}

\makeatletter

\begin{document}

\title{Existence of a unique, nondegenerate solution \\ to parametrized systems \\ of generalized polynomial equations}

\author{Abhishek Deshpande, Stefan M\"uller\textsuperscript{\Letter}}

\maketitle

\begin{abstract}
We consider parametrized systems of generalized polynomial equations (with real exponents) 
in $n$ positive variables, involving $m$ monomials with positive parameters;
that is, $x\in\R^n_>$ such that ${A \, (c \circ x^B)=0}$ with coefficient matrix $A\in\R^{l \times m}$, exponent matrix $B\in\R^{n \times m}$, parameter vector $c\in\R^m_>$ (and componentwise product~$\circ$).

Our main result characterizes the existence of a unique, \green{nondegenerate} solution (up to an exponential manifold) for all parameters 
in terms of the relevant geometric objects of the polynomial system:
the {\em coefficient polytope} and the {\em monomial dependency subspace}.
Technically, we show that unique existence of a \green{nondegenerate} solution is equivalent to a composite (\mono-\expo) map being a \green{diffeomorphism},
and we characterize this property
using Hadamard's global inversion theorem.

Additionally, we provide sufficient conditions 
in terms of sign vectors of the geometric objects,
which represent a genuine multivariate generalization of Descartes' rule of signs for exactly one solution. 
Finally, we illustrate all objects and results in a concrete example.


\vspace{2ex}

\vspace{2ex}
\end{abstract}


\blfootnote{
\scriptsize

\noindent
{\bf Abhishek Deshpande} \\
Center for Computational Natural Sciences and Bioinformatics, 
International Institute of Information Technology Hyderabad,
Hyderabad, Telangana 500032,
India \\[1ex]
{\bf Stefan~M\"uller} \\
Faculty of Mathematics, University of Vienna, Oskar-Morgenstern-Platz 1, 1090 Wien, Austria \\
\Letter \, {\tt st.mueller@univie.ac.at}
}


\section{Introduction} \label{sec:intro}

In this work, we study positive solutions to parametrized systems of generalized polynomial equations (with real exponents). In particular, we characterize the existence of a unique, \green{nondegenerate} solution (modulo an exponential manifold) for all parameters,
and we obtain a corresponding multivariate Descartes' rule of signs (for exactly one solution).
Formally, we consider the parametrized system of generalized polynomial equations 
\begin{equation} \label{eq:system}
A \left( c \circ x^B \right) = 0
\end{equation}
for $x \in \R^n_>$, where $A\in\R^{l \times m}$, $B\in\R^{n \times m}$, and $c \in \R^m_>$.
That is, $A$ is the \emph{coefficient} matrix, $B$ is the \emph{exponent} matrix, and $c$ is the \emph{parameter} vector. 
Explicitly,
there are $l$ equations
\[
\sum_{j=1}^m a_{ij} \, c_j \, x_1^{b_{1j}} \cdots x_n^{b_{nj}} = 0 , \quad i=1,\ldots,l ,
\]
for $n$ positive variables $x_i>0$, $i=1,\ldots,n$,
involving $m$ positive parameters $c_j>0$ and $m$ monomials $x_1^{b_{1j}} \cdots x_n^{b_{nj}}$, $j=1,\ldots,m$.
We obtain the compact form~\eqref{eq:system} as follows.
From the exponent matrix $B = (b^1,\ldots,b^m)$,
we define the monomials $x^{b^j} = x_1^{b_{1j}} \cdots x_n^{b_{nj}} \in \R_>$,
the vector of monomials $x^B \in \R^m_>$ via $(x^B)_j = x^{b^j}$,
and the vector of monomial terms $c \circ x^B \in \R^m_>$ using the componentwise product~$\circ$.

Equation~\eqref{eq:system} accommodates both fewnomial systems (without parameters), cf.~\cite{Khovanskii1991,sottile2011real},
and generalized mass-action systems (with parameters),
cf.~\cite{HornJackson1972,Horn1972,Feinberg1972} and \cite{MuellerRegensburger2012,MuellerRegensburger2014,Mueller2016,MuellerHofbauerRegensburger2019}.
In both fewnomial and reaction network theory,
fundamental problems include the identification of lower and upper bounds 
for the number of positive solutions,
in particular,
multivariate extensions of Descartes' rule of signs
and, as the most basic problems,
the characterization of existence and/or uniqueness of positive solutions. 

{\bf Novelty and significance.}
To clearly delineate the novelty and significance of our results, it is essential to position this work within the multi-dimensional landscape of {\em polynomial systems}. Prior literature classifies this domain along several independent dimensions:
\begin{itemize}
\item 
polynomials vs.\ generalized polynomials: integer vs.\ real exponents
\item 
equations vs.\ inequalities: $A\,( c \circ x^B ) = 0$ vs.\ 
$A\,( c \circ x^B ) \ge 0$
\item 
parametrized vs.\ unparametrized: $A\,( c \circ x^B ) = 0$ vs.\ $A \, x^B  = 0$
\item
homogeneous  vs.\ inhomogeneous: $A\,( c \circ x^B ) = 0$ vs.\ 
$A\,( c \circ x^B ) = z$
\item 
general vs.\ nondegenerate solutions
\item
upper vs.\ lower bounds,
uniqueness vs.\ existence; \\
for given vs.\ for all parameters
\end{itemize}

In this work, we study the ``classical'' case of {\em (homogeneous) parametrized} systems of {\em generalized} polynomial {\em equations}, Equation~\eqref{eq:system}, 
and characterize the {\em existence} of a {\em unique}, {\em nondegenerate} solution (modulo an exponential manifold) {\em for all parameters}.
We first elaborate on our results and then contrast them with prior literature along the dimensions listed above.

We build on recent results by M\"uller and Regensburger~\cite{MuellerRegensburger2023a,MuellerRegensburger2023b},
which allow us to rewrite parametrized polynomial systems in terms of geometric objects, such as the {\em coefficient polytope} and the {\em monomial dependency subspace}.
In particular,
positive solutions to parametrized systems of {\em polynomial} equations (and even inequalities)
are in bijection with solutions to {\em binomial} equations on the coefficient polytope (determined by the dependency subspace). 

As we show, the existence of a unique, nondegenerate solution for all parameters
is equivalent to a composite (\mono-\expo) map being a diffeomorphism.
By Hadamard's global inversion theorem,
this is further equivalent to local invertibility and properness of that map.
We characterize both properties
in terms of the coefficient polytope (and its associated subspace) and the dependency subspace,
and all criteria can be checked effectively.
Finally,
we provide sufficient conditions for unique existence
in terms of sign vectors of the geometric objects involved,
which can be referred to as a {\em genuine} multivariate Descartes' rule of signs
(for exactly one solution for all parameters).
In terms of the dimensions listed above, 
we address problem {\color{blue}(III)} in Table~\ref{tab:dim} below.
In contrast, the {\em first} multivariate Descartes' rules in the literature
addressed problems (i), (ii), (iii) -- at most/at least/exactly one solution -- to {\em inhomogeneous} polynomial equations, see~\cite{Mueller2016,MuellerHofbauerRegensburger2019} based on \cite{CraciunGarcia-PuenteSottile2010,MuellerRegensburger2012}.

\begin{table}[h] 
\begin{center}
\begin{tabular}{|l|l|l|}
\hline
& Inhomogeneous & Homogeneous \\ \hline
Uniqueness
& (i) $\forall c. \forall z. \exists_{\le1} x. \, A\,(c\circ x^B)=z$
& (I) $\forall c. \exists_{\le1} x. \, A\,(c\circ x^B)=0$
\\ \hline
Existence 
& (ii) $\forall c. \forall z. \exists x. \, A\,(c\circ x^B)=z$ 
& (II) $\forall c. \exists x. \, A\,(c\circ x^B)=0$ 
\\ \hline
Unique existence
& (iii) $\forall c. \forall z. \exists! x. \, A\,(c\circ x^B)=z$ 
& {\color{blue}(III) $\forall c. \exists! x. \, A\,(c\circ x^B)=0$}
\\ \hline
\end{tabular}
\end{center}
\caption{Schematic classification of {\em parametrized} systems of {\em (generalized)} polynomial {\em equations}.
This work addresses problem~(III).}
\label{tab:dim} 
\end{table}


Note that (equivalent or sufficient) conditions for the inhomogeneous problems (i), (ii), (iii)
are trivially sufficient for the homogeneous problems (I), (II), (III), but are too strong in general.
In Section~\ref{sec:example} we provide an example,
where the equivalent conditions for (iii) in~\cite{MuellerHofbauerRegensburger2019} are not fulfilled,
whereas our sufficient (sign vector) conditions for (III) do hold.

{\bf Prior literature.}
The injectivity of the generalized polynomial map $f(x) = A\,(c\circ x^B)$
guarantees the uniqueness of solutions to $f(x)=0$
and naturally leads to the study of the inhomogeneous system ${A\,(c\circ x^B) = z}$,
possibly subject to additional linear constraints.
For overviews on injectivity, see~\cite{BanajiPantea2016,FeliuMuellerRegensburger2019}.
In~\cite{Mueller2016}, M\"uller, Feliu, Regensburger, Conradi, Shiu, and Dickenstein 
recognized previous results on problems (i) and (ii) in Table~\ref{tab:dim} 
--- originally established in~\cite{CraciunGarcia-PuenteSottile2010,MuellerRegensburger2012} ---
as first partial multivariate generalizations of Descartes' rule.
Subsequently, M\"uller, Hofbauer, and Regensburger~\cite{MuellerHofbauerRegensburger2019} characterized problem~(iii), 
by applying Hadamard's global inversion theorem to a family of exponential maps,
rather than to the composite (\mono-\expo) map investigated in this work.

Recently, Checa and Feliu~\cite{ChecaFeliu2025} provided sufficient conditions for problem~(I) in Table~\ref{tab:dim}, regarding the uniqueness of positive solutions to the homogeneous system $A(c\circ x^B) = 0$, possibly subject to additional linear constraints. Problem~(II) has not yet been addressed in the literature,
while problem (III) -- for nondegenerate solutions -- is the subject of this work.

For fewnomials, 
that is, unparametrized systems $A \, x^B = 0$, 
Khovanskii established an {\em upper bound} on the number of positive solutions in terms of the number of variables and monomials~\cite{khovanskii1980class},
which was subsequently sharpened by Bihan, Sottile, and Rojas~\cite{bihan2007new,bihan2008sharpness}.
For specific classes of multivariate fewnomials (that can be reduced to univariate problems), 
Bihan, Dickenstein, and Forsg{\aa}rd~\cite{bihan2017descartes,bihan2021optimal} obtained a multivariate Descartes' rule of signs,
while Telek and Feliu~\cite{feliu2022generalizing,telek2023topological} provided upper bounds on the number of connected components of the complement of a single hypersurface.
Regarding {\em lower bounds}, Bihan, Dickenstein, and Giaroli~\cite{bihan2020sign,bihan2020lower} utilized degree theory and Viro's method to establish sufficient conditions for the existence of at least one positive solution.


\subsubsection*{Organization of the work}

In Section~\ref{sec:problem}, we define our problem: the existence of a unique, \green{nondegenerate} solution (modulo an exponential manifold) for all parameters.
To this end, we introduce the relevant geometric objects for polynomial systems (in particular, the coefficient polytope and the monomial dependency subspace)
and recall the main results of previous work in Theorem~\ref{thm:previous}.
Further, we address the \green{degeneracy} of solutions in Theorem~\ref{thm:iso}.
In Section~\ref{sec:reformulation}, we reformulate our problem in three steps:
we parametrize the coefficient polytope via an \expo map;
we introduce a composite (\mono-\expo) map whose \green{bijectivity and local invertibility} are equivalent to the existence of a unique, \green{nondegenerate} solution; 
and we outline the use of Hadamard's global inversion theorem. 
%
In Section~\ref{sec:main}, we present our main results.
We characterize local invertibility and properness of the composite map
(in terms of matrices and geometric objects, respectively),
see~Theorems~\ref{thm:local_invert_1}, \ref{thm:local_invert_2} and \ref{thm:proper_1}, \ref{thm:proper_2}.
Altogether, we provide equivalent conditions for the composite map being a \green{diffeomorphism},
see Theorem~\ref{thm:main}.
Moreover, we provide sufficient conditions,
which represent a multivariate generalization
of Descartes’ rule of sign,
see~Theorem~\ref{thm:suff}.
%
%
In Section~\ref{sec:example}, 
we illustrate all objects and results in a concrete example,
and in Section~\ref{sec:discussion},
we discuss how this work extends previous work on unique existence.



\subsubsection*{Notation}
For two vectors $x, y \in \R^n$,
we denote their scalar product by $x \cdot y \in \R$
and their component-wise (Hadamard) product by $x \circ y \in \R^n$. We write $\id_n \in \R^{n \times n}$ for the identity matrix and $\ones_n \in \R^n$ for the vector with all entries equal to one.

We denote the positive real numbers by $\R_>$ and the nonnegative real numbers by $\R_\ge$. For $x \in \R^n_>$ and $y \in \R^n$, we introduce the monomial $x^y \in \R_>$ as $x^y = \prod_{i=1}^n (x_i)^{y_i}$.
Further, for $Y = (y^1,\ldots,y^m) \in \R^{n \times m}$,
we introduce the vector of monomials $x^Y \in \R^m_>$ via $(x^Y)_j = x^{(y^j)}$.
For $x \in \R^n$, we define $\e^x \in \R^n_>$ componentwise, that is, $(\e^x)_i=\e^{x_i}$,
and for $x \in \R^n_>$, we define $\ln x \in \R^n$ analogously.

For a vector $x \in \R^n$, we obtain its sign vector $\sign(x) \in \{-,0,+\}^n$ by applying the sign function componentwise.
For a subset $S \subseteq \R^n$,
$
\sign(S) = \{ \sign(x) \mid x \in S \} \subseteq \{-,0,+\}^n .
$


\section{Problem definition} \label{sec:problem}

We consider the parametrized system of generalized polynomial equations~\eqref{eq:system}.
For simplicity, we do not partition the system into {\em classes}~\cite{MuellerRegensburger2023a}, but treat it as a single class.
In order to state Theorem~\ref{thm:previous} below, 
we introduce geometric objects and auxiliary matrices as defined in~\cite{MuellerRegensburger2023a,MuellerRegensburger2023b}.

\begin{enumerate}[(i)]

\item 
Let $C = \ker A \cap \R_>^m$ denote the \emph{coefficient cone}. Its closure $\overline{C} = \ker A \cap \R_\ge^m $ is a polyhedral cone, called an s-cone (subspace cone) in~\cite{MuellerRegensburger2016}.
We assume $C \neq \emptyset$ in the following,
as a necessary condition for the existence of solutions to equations~\eqref{eq:system}.

Further, let $\Delta = \{y \in \R^m_\ge \mid v \cdot y =1 \}$ for some $v\in\rm{relint}(C^*)$, where $C^*$ is the dual cone of $C$. For simplicity, we choose $v=\ones_m$. Then, $\Delta$ is the standard simplex in~$\R^m$.

Now, $P = C \cap \Delta$ denotes the \emph{coefficient polytope}.
(Strictly speaking, only its closure $\overline{P}$ is a polytope.)
Note that the coefficient polytope $P$ has dimension one less than the dimension of the coefficient cone $C$.

\item 
Let
\[
\mathcal{B} = \begin{pmatrix}
B \\
\ones_m^\trans
\end{pmatrix}
\quad \text{and} \quad
D = \ker \mathcal{B} ,
\]
using one ``Cayley'' row of 1's 
(again because we assume one class).
$D$ is called the \emph{(monomial) dependency subspace}.
(It records linear dependencies between the columns of $B$,
which define the monomials.)

Analogously, let
\[
\mathcal{A} = \begin{pmatrix}
A \\
\ones_m^\trans 
\end{pmatrix}
\quad \text{and} \quad
T = \ker \mathcal{A} .
\]

Further,
let $\basT \in \R^{m \times \dim T}$ and $H \in \R^{m \times \dim D}$ be basis matrices such that $T = \im \basT$ and $D = \im H$.
Note that $\basT$ and $H$ have zero column sums.

\item 
Let 
$I = \begin{pmatrix} \id_{m-1} \\ -\ones_{m-1}^\trans \end{pmatrix} \in \R^{m \times (m-1)}$ 
denote the ``incidence matrix'' 
and 
\[ M = B \, I \in\R^{n\times(m-1)} \] 
be the resulting monomial difference matrix. 
(The latter is generated by taking the differences between the first $m-1$ columns of $B$ and its last column.)
Now, let \[ L = \im M \] denote the \emph{(monomial) difference subspace}
and 
$d = \dim(\ker M)$ denote the \emph{(monomial) dependency}.
By~\cite[Proposition 1 and Lemma 4]{MuellerRegensburger2023a},
\[
d = \dim D = m - 1 - \dim L .
\]

\item 
Finally, let $E = I M^* \in \R^{m \times n}$, where $I$ is the incidence matrix and $M^* \in \R^{(m-1) \times n}$ is a generalized inverse of $M$.

\end{enumerate}

We can now summarize the main results of our previous work~\cite{MuellerRegensburger2023a} for generalized polynomial inequalities (restricted to equations, cf.~\cite{MuellerRegensburger2023b}).

\begin{theorem}[\cite{MuellerRegensburger2023a}, Theorem 5 and Proposition 6] \label{thm:previous}
Consider the parametrized system of generalized polynomial equations $A \, (c \circ x^B) = 0$. The solution set $Z_c = \{ x \in \R^n_> \mid A \, (c \circ x^B) = 0 \}$ can be written as
\[
Z_c = \{ (y\, \circ\, c^{-1})^E \mid y \in Y_c \} \circ \e^{L^\perp} ,
\]
where
\[
Y_c = \{ y \in P \mid y^z = c^z \text{ for all } z\in D \} 
\]
is the solution set on the coefficient polytope $P$.
Moreover, there is a bijection between $Z_c / \e^{L^\perp}$ and $Y_c$.
\end{theorem}

Theorem~\ref{thm:previous} can be read as follows: 
In order to determine the solution set $Z_c$ (to {\em polynomial equations} on the positive orthant), 
first determine the solution set $Y_c$ (to {\em binomial equations} on the coefficient polytope).
Recall that the coefficient polytope $P$ is determined by the coefficient matrix $A$,
and the dependency subspace $D$ is determined by the exponent matrix $B$.

To a solution $y \in Y_c$,
there corresponds a representative solution $x = (y \circ c^{-1})^E \in Z_c$. 
In fact, if (and only if) $\dim L < n$, 
then $y \in Y_c$ corresponds to an exponential manifold of solutions, $x \circ \e^{L^\perp} \subseteq Z_c$.
Strictly speaking, existence of a unique solution corresponds to $|Y_c|=1$ and $\dim L = n$
(that is, $L^\perp = \{0\}$).

\green{{\bf Problem definition.}
In this work, we characterize when $|Z_c / \e^{L^\perp}| = 1$ and all solutions are nondegenerate (have injective Jacobian matrices) for all parameters~$c$.
In fact, we characterize when $|Y_c|=1$ and all solutions (on the coefficient polytope) are nondegenerate for all parameters $c$.
This is equivalent,
since $y \in Y_c$ and a corresponding solution $x \in Z_c$ are nondegenerate at the same time, see Theorem~\ref{thm:iso} below.}

First, note that, to any $x$ in the positive orthant (not necessarily a solution),
there corresponds a vector $\bar y=c \circ x^B$ in the coefficient cone.
This, in turn, defines a proportional vector $y$ in the coefficient polytope
(with $y = \alpha \, \bar y$ and $\alpha>0$),
which can be uniquely represented by a parameter~$\xi$ such that
\[
c \circ x^B = \underbrace{\bar y}_{\in C} \sim \underbrace{y}_{\in P} = p(\xi) \quad \text{(e.g. } p(\xi) = y^* + \basT \xi ).
\]
In the last step, 
a parametrization $p$ of the coefficient polytope $P$ is specified. 
In the example above, $p$ is a linear map with base point $y^* \in P$ and basis matrix~$\basT$ for $\ker \mathcal{A}$, the linear subspace associated with the affine hull of the polytope.
For the moment, we just assume $p$ to be an injective immersion
(that is, $p$ is injective and its Jacobian matrix $J_p$ has full column rank).

Second, $y^z = c^z$ for all $z\in D$ can be written as $y^H-c^H=0$ with a basis matrix~$H$ for $D = \ker \mathcal{B}$, the dependency subspace.

Third,
the 
map
$P \to \R^n_>, \, y \mapsto (y \circ c^{-1})^E$,
assigns to a solution $y$ in the polytope 
a representative solution $x$. 

Now, we can state the result regarding nondegeneracy announced above.
For a corresponding result in the context of reaction networks,
see~\cite[Theorem~3.14(iii)]{BF2026}.

\begin{theorem} \label{thm:iso}
For fixed parameters $c$, consider the maps
\[
f(x) = A \left( c \circ x^B \right) , \quad
g(y) = y^H - c^H , \quad
\phi(y) = \left( y \circ c^{-1} \right)^E ,
\]
with $x \in \R^n$ and $y \in P$.
Let $p$ be a parametrization of the coefficient polytope~$P$
that is an injective immersion (i.e., $p$ is injective and its Jacobian~$J_{p}$ has full column rank),
and define the composite maps
\[
\bar g(\xi) = g(p(\xi)) = p(\xi)^H - c^H , \quad 
\bar \phi(\xi) = \phi(p(\xi)) = \left( p(\xi) \circ c^{-1} \right)^E .
\]
The associated Jacobian matrices, evaluated at $x$ and $\xi$, respectively, are given by
\begin{gather*}
J_f = A \diag\left( c \circ x^B \right) B^\trans \diag(x^{-1}), \quad
J_{\bar g} = \diag\left( p(\xi)^H \right) H^\trans \diag\left( p(\xi)^{-1} \right) J_p , \\
J_{\bar \phi} = \diag\left( \bar\phi(\xi) \right) E^\trans \diag\left( p(\xi)^{-1} \right) J_p .
\end{gather*}
For corresponding solutions $x$ and $\xi$,
that is, $f(x)=0$, $\bar g(\xi)=0$, and $x=\bar\phi(\xi)$,
the map $J_{\bar\phi}$ induces the isomorphism
\[
\ker J_{f} / \diag(x) L^\perp 
\cong 
\ker J_{\bar g} .
\]
\end{theorem}

\begin{proof}
Using $\bar y = c \circ x^B$ and $y=p(\xi)$ as above,
we first simplify the notation for the Jacobian matrices,
\begin{gather*}
J_f = A \diag(\bar y) B^\trans \! \diag(x^{-1}) , \quad
J_{\bar g} = \diag( y^H ) H^\trans \! \diag( y^{-1} ) J_p , \\
J_{\bar\phi} = \diag(x) E^\trans \! \diag( y^{-1} ) J_p .
\end{gather*}
Using the ``incidence matrix'' $I$ (with $\ker I = \{0\}$ and $\ker I^\trans = \im \ones_m$) and the monomial difference matrix $M = B I$,
both defined in (iii) at the beginning of this section,
we have the equivalences
\begin{align*}
J_{\bar g} \, w = 0
& \iff \diag( y^H ) H^\trans \! \diag( y^{-1} ) J_p \, w = 0 \\
& \iff \diag( y^{-1} ) J_p \, w \in \ker H^\trans = \im \mathcal B^\trans \\
& \iff \exists u, \la \colon \diag( y^{-1} ) J_p \, w = \textstyle \begin{pmatrix} B^\trans & \ones_m \end{pmatrix} \binom{u}{\la} \\
& \iff \exists u \colon I^\trans \diag( y^{-1} ) J_p \, w = M^\trans u \\ 
& \iff \exists v \colon I^\trans \diag( y^{-1} ) J_p \, w = M^\trans \diag(x^{-1}) v 
\end{align*}
and
\begin{align*}
J_f \, v = 0 
& \iff A \diag(\bar y) B^\trans \! \diag(x^{-1}) v = 0 \\
& \iff  B^\trans \! \diag(x^{-1}) v \in \ker \left( A \diag(\bar y) \right) \\
& \overset{\text{Lemma~\ref{lem:iso}}}{\iff}  B^\trans \! \diag(x^{-1}) v \in \diag(y^{-1}) \im J_p + \im \ones_m \\
& \iff \exists w, \la \colon B^\trans \! \diag(x^{-1}) v = \diag(y^{-1}) J_p \, w + \ones_m \la \\
& \iff \exists w \colon M^\trans \diag(x^{-1}) v = I^\trans \diag( y^{-1} ) J_p \, w . 
\end{align*}
The relation
\begin{equation} \label{eq:link} \tag{$\star$}
\underbrace{I^\trans \diag( y^{-1} ) J_p}_{Q_{\bar g}} w 
= 
\underbrace{M^\trans \diag(x^{-1})}_{Q_f} v ,
\end{equation}
appearing at the end of both chains of equivalences,
implies that the images of the two kernels (of the respective Jacobian matrices) coincide,
\[
Q_{\bar g} (\ker J_{\bar g}) = Q_f (\ker J_f) .
\]
By the homomorphism theorem for vector spaces (that is, $R(S) \cong S/(S\cap\ker R)$ for a matrix $R$ and a vector space $S$),
\[
Q_f (\ker J_f)
\cong
\ker J_f / (\ker J_f \cap \ker Q_f) 
, \quad
Q_{\bar g} (\ker J_{\bar g})
\cong
\ker J_{\bar g} / (\ker J_{\bar g} \cap \ker Q_{\bar g}) 
\]
and hence
\[
\ker J_f / (\ker J_f \cap \ker Q_f) 
\cong
\ker J_{\bar g} / (\ker J_{\bar g} \cap \ker Q_{\bar g}) .
\]
It remains to simplify the isomorphism.

First, $\ker Q_f \subseteq \ker J_f$ (and hence $\ker J_f \cap \ker Q_f = \ker Q_f$).
This can be seen follows:
if $Q_f \, v = 0$, then the relation~\eqref{eq:link} holds for $w=0$,
that is, the last line in the equivalences for $J_f \, v = 0$ holds.

Second, 
$\ker Q_f = \diag(x) \ker M^\trans = \diag(x) L^\perp$
since $\ker M^\trans = (\im M)^\perp = L^\perp$.

Third, $\ker Q_{\bar g} = \{0\}$ (and hence $\ker J_{\bar g} \cap \ker Q_{\bar g} = \{0\}$).
This can be seen follows:
by Lemma~\ref{lem:Jp}, $\ker J_p = \{0\}$ and $\im J_p = \ker \mathcal{A} \subseteq \ker \ones_m^\trans$.
However, $\ker I^\trans = \im \ones_m$
and hence $\ker (I^\trans \diag(y^{-1})) \cap \im J_p = \{0\}$.

Finally, after multiplying \eqref{eq:link} with a generalized inverse $(M^*)^\trans$ and $\diag(x)$ from the left,
we have $\diag(x) (M^*)^\trans I^\trans \diag( y^{-1} ) J_p \, w = J_{\bar \phi} \, w = v$.
\end{proof}

In the proof of Theorem~\ref{thm:iso}, we have used the following results.

\begin{lemma} \label{lem:Jp}
Let $U\subset \R^m$ and $p \colon U \to P$ be a parametrization of the coefficient polytope $P$
that is an injective immersion.
Then, for all $\xx \in U$, $\ker J_p(\xx) = \{0\}$ and $\im J_p(\xx) = \ker \mathcal{A}$.
\end{lemma}
\begin{proof}
Since the map $p$ is an injective immersion, 
$\ker J_p = \{0\}$.
Further,
$\im J_p \subseteq \ker \mathcal{A}$,
that is,
the tangent space of the parametrization is contained in the subspace associated with the affine hull of the polytope.
By the rank-nullity theorem for $J_p \in \R^{m \times \dim P}$, $\dim(\im J_p) = \dim P$.
Since $\im J_p \subseteq \ker \mathcal{A}$ and $\dim(\im J_p) = \dim P = \dim (\ker \mathcal A)$,
we get $\im J_p = \ker \mathcal A$.
\end{proof}

\begin{lemma} \label{lem:iso}
Let $A \in \R^{l \times m}$ be the coefficient matrix.
Further, let $\bar y \in \R^m_>$ such that $A \bar y = 0$.
Finally, let $p$ be a parametrization of the coefficient polytope $P$
that is an injective immersion.
Then, $\ker A = \im J_p + \im \bar y$, that is,
\[ \ker \left( A \diag(\bar y) \right) 
= \im \left( \diag(\bar y^{-1}) J_p \right)
+ \im \ones_m . \]
\end{lemma}

\begin{proof}
Since $A \bar y = 0$ with $\bar y \in \R^m_>$, the vector $\ones_m$ is not in the row span of~$A$.
In particular, for $\mathcal{A} = \binom{A}{\ones_m^\trans}$,
$\rank \mathcal A = \rank A + 1$.
By the rank-nullity theorem (for $A$ and $\mathcal{A}$), $\dim (\ker A) = \dim (\ker \mathcal{A}) + 1$.
Indeed, $\ker A = \ker \mathcal{A} + \im \bar y$,
since $\bar y \in \ker A$, but $\bar y \notin \ker \mathcal{A}$.
By Lemma~\ref{lem:Jp}, $\im J_p = \ker \mathcal{A}$.
Hence, $\ker A = \im J_p + \im \bar y$.
The displayed equation follows by multiplying with $\diag(\bar y^{-1})$ from the left.
\end{proof}


\section{Problem reformulation} \label{sec:reformulation} 

We reformulate the problem in three steps: 
(i) we parametrize the coefficient polytope via an \expo map;
\green{(ii) we define a map~$\map$ that is a diffeomorphism exactly when $|Y_c|=1$ and $y\in Y_c$ is nondegenerate for all~$c$;}
and (iii) we outline how we are going to apply our main technical tool, Hadamard's global inversion theorem, to the map.


\subsection{A parametrization of the coefficient polytope}

From toric geometry~\cite{Fulton1993}, we have the following result.

\begin{proposition}[\cite{Fulton1993}, page 83] \label{pro:moment}
Let $U = (u^1, u^2,\ldots, u^{r}) \in \R^{n \times r}$ 
and let the polytope $P = \conv(U)$ 
have full dimension~$n$.
Further, let $\varepsilon \in \R^r_>$. 
Then, 
\[
\pi \colon \R^n \to \interior{P}, \quad
x \mapsto \frac{ \sum_{k=1}^{r} \varepsilon_k \e^{u^k \cdot x} u^k }{ \sum_{k=1}^{r} \varepsilon_k \e^{u^k \cdot x}}
\]
is a real analytic isomorphism.
\end{proposition}

We call $\pi$ an {\em \expo map} to distinguish it from an algebraic moment map,
where the exponential weights $\e^{u^k \cdot x}$ are replaced by monomial weights $x^{u^k}$ (and $x \in \R^n$ is replaced by $x \in \R^n_>$).

We extend the result to polytopes that do not have full dimension.

\begin{proposition} \label{pro:moment_extended}
Let $U = (u^1, u^2,\ldots, u^{r}) \in \R^{n \times r}$ 
and let the polytope $P = \conv(U)$
have dimension~$d \le n$.
In particular, $P$ lies in an affine subspace with associated linear subspace $S \subseteq \R^n$ of dimension $d$.
Let $S = \im Q$ with $Q \in \R^{n \times d}$ and $Q^\trans Q = \id$.
Finally, let $\varepsilon \in \R^r_>$. 
Then, 
\[
\pi \colon \R^d \to \relint{P}, \quad
x \mapsto \frac{ \sum_{k=1}^{r} \varepsilon_k \e^{u^k \cdot Q x} u^k }{ \sum_{k=1}^{r} \varepsilon_k \e^{u^k \cdot Q x}}
\]
is a real analytic isomorphism.
\end{proposition}

\begin{proof}
The polytope $P$ lies in an affine subspace $u^0 + S$ with $u^0 \in \R^n$.
We introduce the affine bijection $a \colon \R^d \to u^0 + S$, $x \mapsto Q x + u^0$
with inverse $a^{-1} \colon u^0 + S \to \R^d$, $u \mapsto Q^\trans(u-u^0$).
(Note that $Q$ and $Q^\trans$ are generalized inverses of each other.) 

Let $\bar U = (\bar u^1, \bar u^2,..., \bar u^{r}) \in \R^{d \times r}$ with 
$\bar u^k = a^{-1}(u^k) = Q^\trans (u^k - u^0)$
(and hence $Q \, \bar u^k = u^k - u^0$)
and $\bar P = \conv(\bar U)$. Then $\bar P$ has dimension $d$. 
By Proposition~\ref{pro:moment}, 
\[
\bar \pi \colon \R^d \to \interior{\bar P}, \quad x \mapsto 
\frac{ \sum_{k=1}^{r} \varepsilon_k \e^{\bar u^k \cdot x} \bar u^k }{ \sum_{k=1}^{r} \varepsilon_k \e^{\bar u^k \cdot x}} 
\]
is a real analytic isomorphism. 
Now,
\begin{align*}
a(\bar \pi(x)) &= Q \, \frac{ \sum_{k=1}^{r} \varepsilon_k \e^{\bar u^k \cdot x} \bar u^k }{ \sum_{k=1}^{r} \varepsilon_k \e^{\bar u^k \cdot x}} + u^0 \\
&= \frac{ \sum_{k=1}^{r} \varepsilon_k \e^{Q^\trans(u^k-u^0) \cdot x} \, (u^k-u^0) }{ \sum_{k=1}^{r} \varepsilon_k \e^{Q^\trans(u^k-u^0) \cdot x}} + u^0 \\
&= \frac{ \sum_{k=1}^{r} \varepsilon_k \e^{(u^k-u^0) \cdot Q x} \, (u^k-u^0) }{ \sum_{k=1}^{r} \varepsilon_k \e^{(u^k-u^0) \cdot Q x}} + u^0 \\
&= \frac{ \sum_{k=1}^{r} \varepsilon_k \e^{u^k \cdot Q x} \, (u^k-u^0) }{ \sum_{k=1}^{r} \varepsilon_k \e^{u^k \cdot Q x}} + u^0 \\
&= \frac{ \sum_{k=1}^{r} \varepsilon_k \e^{u^k \cdot Q x} u^k }{ \sum_{k=1}^{r} \varepsilon_k \e^{u^k \cdot Q x}} \\
&= \pi(x) .
\end{align*}
Since both $a$ and $\bar \pi$ are real analytic isomorphisms, also the composition $\pi = a \circ \bar \pi$ is a real analytic isomorphism.
\end{proof}

The coefficient polytope $P$  with dimension $\dP$ lies in an affine subspace with associated linear subspace $T = \ker \mathcal{A} = \im \basT$ with dimension $\dP$.
Now,
let $\overline P$ have vertices $v^1, v^2, \ldots, v^{\nP} \in \R^m$,
in particular,  
$\overline P = \conv(V)$ with $V = (v^1, v^2, \ldots, v^{\nP}) \in \R^{m \times \nP}$.
By Proposition~\ref{pro:moment_extended} (with $U=V$, $\varepsilon = \ones_{\nP}$, and $Q=\basT$ assumed to be orthonormal),
$P$~can be parametrized by the real analytic isomorphism
\begin{equation} \label{eq:param}
p \colon \R^{\dP} \to P, \quad
\xx \mapsto \frac{ \sum_{k=1}^{\nP} \e^{v^k \cdot \basT \xx} v^k }{ \sum_{k=1}^{\nP} \e^{v^k \cdot \basT \xx}} .
\end{equation}


\subsection{The composite map \texorpdfstring{$\map$}{pdf}}

Recall $D = \ker \mathcal{B} = \im H$ with $H \in \R^{m \times d}$.
The solution set on the coefficient polytope can be written as
\[
Y_c = \{ y \in P \mid y^z = c^z \text{ for all } z \in D \} 
= \{ y \in P \mid y^H = c^H \} ,
\]
which involves the monomial map
\[
h \colon \R^m_> \to \R^d_>, \quad
y \mapsto y^{H} .
\]
Since $H$ has zero column sums,
$h$ is homogeneous of degree zero,
that is, $h(\alpha y) = h(y)$ for $\alpha \in \R_>$.

On the one hand, 
every $y \in P$ is a solution to $y^H=c^H$ for some~$c$ (just choose $c=y$).
On the other hand, 
the set of right-hand sides for all~$c$ amounts to $\{c^H \mid c \in \R^m_>\} = \R^d_>$, since $H$ has full rank $d$.
As a consequence, $|Y_c|=1$ for all~$c$ 
if and only if
$h |_P$ (the restriction of $h$ to $P \subset \R^m_>$) is bijective.
Note that, for $h \colon P \to \R^d_>$ to be bijective,
the dimensions $d_P = \dim P$ and $d$ must agree.

\green{To characterize the nondegeneracy of $y \in Y_c$ (that is, of $y \in P$ with $y^H=c^H$),
we use the parametrization $p \colon \R^{\dP} \to P$ of the polytope $P$ 
given in Equation~\eqref{eq:param}.
That is, we consider $\xx \in \R^{\dP}$ with $p(\xx)^H = c^H$.}

\begin{definition} 
Let $v^1,v^2, \ldots ,v^{\nP}$ denote the vertices of the coefficient polytope~$P$. We introduce the map
\begin{equation} \label{eq:F}
\map \colon \R^{\dP} \to \R^d_>, \quad
\xx \mapsto p(\xx)^H = \left(\frac{\sum_{k=1}^{\nP} \e^{ v^k \cdot \basT \xx} v^k }{ \sum_{k=1}^{\nP} \e^{v^k \cdot \basT \xx}}\right)^{H} .
\end{equation}
\end{definition}
That is, $\map(\xi) = h(p(\xi))$.
Since $h$ is homogeneous of degree zero,
we can simplify the definition as
\begin{equation} \label{eq:F_simple}
\map \colon \R^{\dP} \to \R^d_>, \quad
\xx \mapsto \left(\sum_{k=1}^{\nP} \e^{ v^k \cdot \basT \xx } v^k \right)^{H}.
\end{equation}

Now we can reformulate the problem defined in Section~\ref{sec:problem}.

\begin{proposition} \label{pro:reform}
$|Y_c|=1$
and $y \in Y_c$ is \green{nondegenerate} for all~$c$
if and only if 
the map~$\map$ defined in Equation~\eqref{eq:F} is bijective
and $\ker J_\map (\xi) = \{0\}$ for all $\xi \in \R^{d_P}_>$.
\end{proposition}

\begin{proof}
By the argument above,
$|Y_c|=1$ for all~$c$ is equivalent to $h |_P$ being a bijection.
Since the parametrization $p \colon \R^{\dP} \to P$ of the coefficient polytope $P$ given in Equation~\eqref{eq:param} is a real analytic isomorphism,
$h |_P $ is bijective 
if and only if 
$\map$ is bijective. 
Clearly, a solution $\xx \in \R^{\dP}$ to $\map(\xi) = p(\xx)^H = c^H$ has Jacobian matrix $J_\map (\xi)$.
It is nondegenerate if its Jacobian matrix is injective.
\end{proof}

As indicated above, for $\map$ to be bijective, the dimensions of domain and range must match,
that is, $\dP=d$,
which we assume in the following.
Then, $\ker J_\map = \{0\}$ is equivalent to $\det J_\map \neq 0$,
and we can show that the conditions on $\map$ in Proposition~\ref{pro:reform} 
are equivalent to $\map$ being a \green{diffeomorphism}.

\begin{proposition} \label{pro:diff} 
Let $U,V \subset \mathbb{R}^n$ be open, and let $\Psi \colon U \to V$ be $C^1$.  
Then,
\[
\Psi \text{ is a diffeomorphism } \iff \Psi \text{ is bijective and } \det J_\Psi (x) \neq 0 \text{ for all } x\in U.
\]
\end{proposition}

\begin{proof}
($\Rightarrow$) Let $\Psi $ be a diffeomorphism. Then it is bijective and has a $C^1$ inverse.  
By the chain rule, we have
$J_{\Psi^{-1}}(\Psi(x)) J_\Psi(x)=\id$ for every $x\in \R^n$. Therefore, $J_\Psi (x)$ is invertible for all $x$ and hence $\det J_\Psi(x) \neq 0$.  

($\Leftarrow$) Let $\Psi $ be bijective and $\det J_\Psi(x) \neq 0$ for all $x\in U$.

Let $y \in V$ and $x = \Psi^{-1}(y)$.
By the Inverse Function Theorem (since ${\det J_\Psi(x) \neq 0}$), 
there exists a neighborhood $U_x \subset U$ of $x$ such that 
    \[
    \Psi|_{U_x} \colon U_x \to V_y
    \] 
    is a $C^1$-diffeomorphism onto some neighborhood $V_y \subset V$ of $y$. 
    Let 
    \[
    \Psi'_x = (\Psi|_{U_x})^{-1} : V_y \to U_x
    \]
    be the local inverse.  
    Since $\Psi$ is injective and hence $\Psi^{-1}$ is the unique global inverse,
    \[
    \Psi'_x(y) = \Psi^{-1}(y) \text{ for all } y \in V_y.
    \]
Thus, $\Psi^{-1}$ is $C^1$ on $V_y$.
Since $y \in V$ was arbitrary, $\Psi^{-1}$ is $C^1$ on $V$.
\end{proof}

Indeed, the map~$\map$ defined in Equation~\eqref{eq:F} 
satisfies the conditions (on $\Psi$) in Proposition~\ref{pro:diff},
and we will characterize when this map is a \green{diffeomorphism}.

Note that $\map$ depends on the basis matrices $\basT$ and $H$ of the linear subspaces $T$ and $D$,
but ultimately we aim for characterizations in terms of subspaces rather than matrices.


\subsection{Hadamard's global inversion theorem}

To characterize when the map~$\map$ defined in Equation~\eqref{eq:F} is a \green{diffeomorphism},
we employ Hadamard's global inversion theorem in the form of Theorem~\ref{thm:hadamard_maps} (cf. Theorem B in~\cite{gordon1972diffeomorphisms} or Satz II in~\cite{banach1934mehrdeutige}).

Recall that a map $\Psi \colon X \to Y$ is called \emph{proper} if, for every compact subset $K$ of $Y$, the preimage $\Psi^{-1}(K)$ is a compact subset of $X$. 
Further, let $J_{\Psi}$ denote the Jacobian matrix of a map $\Psi$.

\begin{theorem} \label{thm:hadamard_maps}
Let $U \subseteq \R^n$ be open and convex. A $C^1$-map $\Psi \colon \R^n \to U$ is a diffeomorphism if and only if 
\begin{enumerate}[(i)]
    \item $\det(J_{\Psi}(x)) \neq 0$ for all $x\in\R^n$, i.e., $\Psi$ is locally invertible and
    \item $\Psi$ is proper.
\end{enumerate}
\end{theorem}

The properness of a continuous map can be determined by considering unbounded sequences,
cf.~Lemma~10 in~\cite{MuellerHofbauerRegensburger2019}.
For the convenience of the reader, we provide a streamlined proof.

\begin{lemma} \label{lem:proper}
Let $U \subseteq \R^n$ be open.
A continuous map $\Psi \colon \R^n \to U$ is proper if
and only if
$\Psi(x_k t_k) \to y$ implies $y \in \partial U$
for all sequences 
\begin{itemize}
\item
$(x_k)_{k \in \N}$ in $\R^n$ with $|x_k| = 1$ and $x_k \to x$ and \item 
$(t_k)_{k \in \N}$ in $\R_>$ with $t_k \to \infty$.
\end{itemize}
\end{lemma}
\begin{proof}
Let $\Psi(x_k t_k) \to y$, but $y \in U$,
and take a closed ball $K \subset U$ around $y$.
Then $\Psi^{-1}(K)$ contains the unbounded sequence $(x_k t_k)_{k \ge N}$ for some $N \in \N$ and hence is not compact, that is, $\Psi$ is not proper.

Conversely, assume $\Psi(x_k t_k) \to y$ implies $y \in \partial U$ (for all sequences),
let $K$ be a compact subset of $U$,
and show that $\Psi$ is proper,
that is, $\Psi^{-1}(K)$ is compact.
Since $\Psi^{-1}(K)$ is closed,
it suffices to show that every sequence $X_k$ in $\Psi^{-1}(K)$ has a bounded subsequence. 
Assume the contrary; then $|X_k| \to \infty$. 
Since ${\Psi(X_k) \in K}$, there is a subsequence (call it $X_k$ again) such
that $\Psi(X_k) \to y \in K$. Now there is a subsubsequence (call it $X_k$ again) such that $x_k = X_k /|X_k| \to x$, that is, the sequence $x_k$ on the unit sphere converges.
With $t_k = |X_k|$, we have $\Psi(x_k t_k) \to y \in K \subset U$, a contradiction.
\end{proof}


\section{Main results} \label{sec:main}

By Propositions~\ref{pro:reform}
and~\ref{pro:diff},
there exists a unique, \green{nondegenerate} solution in the coefficient polytope for all parameters 
if and only if 
the map~$\map$ defined in Equation~\eqref{eq:F} is a \green{diffeomorphism}.
By Theorem~\ref{thm:hadamard_maps} (Hadamard's global inversion theorem), 
this is further equivalent to $\map$ being (i) locally invertible and (ii) proper.
In Subsections~\ref{subsec:loc_inv} and \ref{subsec:proper},
we characterize local invertibility and properness, respectively,
in terms of the relevant geometrical objects,
and in \ref{subsec:diffeo}, we combine our results to characterize when $\map$ is a diffeomorphism.
Finally, in Subsection~\ref{subsec:suf_con},
we provide sufficient conditions 
in terms of sign vectors
that can be viewed as a {\em genuine} multivariate Descartes' rule for 
exactly one solution.


\subsection{Local invertibility of \texorpdfstring{$\map$}{pdf}} \label{subsec:loc_inv}

We characterize the local invertibility of the map~$\map$ in terms of matrices and geometric objects,
respectively, see Theorems~\ref{thm:local_invert_1} and~\ref{thm:local_invert_2} below.

To this end,
we first compute the Jacobian matrix $J_\map$ of the map $\map(\xi)=h(p(\xi))$ defined in Equation~\eqref{eq:F},
\[
J_\map(\xx) = J_h (p(\xx)) \, J_p (\xx) .
\]
Recall $h(y) = y^H$, in particular, $h_i(y) = \prod_{l=1}^m (y_l)^{H_{li}}$. Hence,
\begin{align*}
\DD{h_i}{y_j} &= \prod_{l \neq j} (y_l)^{H_{li}} \cdot H_{ji} \, (y_j)^{H_{ji} - 1} \\
&= h_i(y) \, H_{ji} \, (y_j)^{-1} ,
\end{align*}
that is,
\[
J_h (y) = \diag (h(y)) \, H^\trans \diag (y^{-1}) .
\]

Now, we can provide a first characterization of local invertibility.
\begin{theorem} \label{thm:local_invert_1}
For the map~$\map$ defined in Equation~\eqref{eq:F},
the following statements are equivalent:
\begin{enumerate}[(i)]
\item
The map~$\map$ is locally invertible. \\
(That is, 
$J_\map(\xx)$ is injective,
for all $\xx$.)
\item The matrix
$
H^\trans \diag((p(\xx)^{-1}) \, \basT 
$
is injective, 
for all $\xx$.
\end{enumerate}
\end{theorem}
\begin{proof}
We show that $\neg$(i) and $\neg$(ii) are equivalent.

Since the parametrization $p$ given in Equation~\eqref{eq:param} is a real analytic isomorphism (and hence an injective immersion), $\ker J_p(\xx) = \{0\}$ and $\im J_p(\xx) = \ker \mathcal{A}$ for all~$\xx$,
by Lemma~\ref{lem:Jp}.
Using $\ker \mathcal{A} = \im G$,
the existence of $\xx$ and a nonzero $\alpha \in \R^{\dP}$ such that 
\[
0 = J_\map(\xx) \, \alpha = J_h(p(\xx)) \, J_p(\xx) \, \alpha
= \diag(h(p(\xx))) \, H^\trans \diag(p(\xx)^{-1}) J_p(\xx) \, \alpha
\]
is equivalent to the existence of $\xx$ and a nonzero $\beta \in \R^{\dP}$ such that
\[
H^\trans \diag(p(\xx)^{-1}) \, \basT \beta = 0 .
\]
\end{proof}

In fact, we can characterize local invertibility without using the parametrization of the coefficient polytope (via the map $p$) and the basis matrices $\basT$ and $H$ (of the linear subspaces $T$ and $D$),
but just in terms of the polytope and the subspaces.

\begin{theorem} \label{thm:local_invert_2}
For the map~$\map$ defined in Equation~\eqref{eq:F},
the following statements are equivalent:
\begin{enumerate}[(i)]
\item
The map~$\map$ is {\em not} locally invertible. 
\item 
There exist $y \in P$ and nonzero $t \in T$ and $\bar d \in D^\perp$ such that $t = y \circ \bar d$.
\end{enumerate}
\end{theorem}
\begin{proof}
By Theorem~\ref{thm:local_invert_1},
(i) is equivalent to the existence of $\xx$ and a nonzero~$\beta$ such that $H^\trans \diag(p(\xx)^{-1}) \, \basT \beta = 0$.
Since $y = p(\xx)$ and $\xx$ are in one-to-one correspondence by the map~$p$,
this is further equivalent to the existence of $y \in P$ and a nonzero~$\beta$ such that 
\[ H^\trans \diag(y^{-1}) \, \basT \beta = 0 . \]
Since $\im \basT = T$ and $\ker H^\trans = D^\perp$,
this is further equivalent to the existence of $y \in P$ and nonzero~$t \in T$ and $\bar d \in D^\perp$ such that
\[
\diag(y^{-1}) \, t = \bar d ,
\]
that is, $t = y \circ \bar d$.
\end{proof}


\subsection{Properness of \texorpdfstring{$\map$}{pdf}} \label{subsec:proper}

We characterize the properness of the map~$\map$ in terms of matrices and geometric objects,
respectively, see Theorems~\ref{thm:proper_1} and~\ref{thm:proper_2} below.

In Lemma~\ref{lem:proper_rays}, 
we first show that the map~$\map$ defined in Equation~\eqref{eq:F} is proper if and only if it is ``proper along rays''. 
Hence, for $\xv \in \R^{\dP}$ (with $|\xv|=1$), we consider $\map(\xv t)$ as a function of $t>0$ (in particular, for $t \to \infty$).

In the form of Equation~\eqref{eq:F_simple}, $\map$ is given by
\[
\map \colon \R^{\dP} \to \R^d_>, \quad
\xv \mapsto \left( \sum_{k=1}^{\nP} \e^{ G^\trans v^k \cdot \xv} v^k \right)^{H} ,
\]
and for $i \in \{1,\ldots,d\}$,
\[ 
\map_i(\xv t) = \prod_{j=1}^m \left( \sum_{k=1}^{\nP} \e^{G^\trans v^k \cdot \xv t} v_j^k \right)^{H_{ji}} 
= \prod_{j=1}^m \left( \sum_{k=1}^{\nP}  \e^{ \mu^k t} v_j^k \right)^{H_{ji}} ,
\]
where we introduced $\mu^k(\xv) := G^\trans v^k \cdot \xv$, for $k=1,\ldots,\nP$.

Now, for $j=1,\ldots,m$,
let 
\begin{equation} 
\mu_j^{\rm{max}}(\xv) := \max_{k\colon v_j^k\neq 0} \mu^k (\xv) = \max_{k\colon v_j^k\neq 0} \, G^\trans v^k \cdot \xv.  
\end{equation}
Then,
\[
\begin{aligned}
\map_i(\xv t)  &= \prod_{j=1}^m \left(
\sum_{k\colon\mu^k < \mu_j^{\rm{max}}} \e^{ \mu^k t}v^k_j +  \sum_{k\colon\mu^k =\mu_j^{\rm{max}}} \e^{ \mu^k t} v^k_j\right)
^{H_{ji}} \\
&= \prod_{j=1}^m \left[\e^{\mu^{\rm{max}}_j t}\left(
\sum_{k\colon\mu^k < \mu_j^{\rm{max}}} \e^{(\mu^k -\mu^{\rm{max}}_j) t} v^k_j +  \sum_{k\colon\mu^k =\mu_j^{\rm{max}}}  \, v^k_j\right)\right]
^{H_{ji}} \\
&= \prod_{j=1}^m \e^{\mu^{\rm{max}}_j H_{ji} \, t} \prod_{j=1}^m\left(
\sum_{k\colon\mu^k < \mu_j^{\rm{max}}}  \e^{(\mu^k -\mu^{\rm{max}}_j) t} v^k_j + \sum_{k\colon\mu^k =\mu_j^{\rm{max}}}  \, v^k_j\right)
^{H_{ji}} \\
&= 
\e^{\sum_{j=1}^m \mu^{\rm{max}}_j H_{ji} \, t} \prod_{j=1}^m \left(
\sum_{k\colon\mu^k < \mu_j^{\rm{max}}} \e^{(\mu^k -\mu^{\rm{max}}_j) t} v^k_j +  \sum_{k\colon\mu^k =\mu_j^{\rm{max}}}  \, v^k_j \right)^{H_{ji}} .
\end{aligned}
\]

That is,
we have the limit
\begin{equation} \label{eq:limit}
\map_i(\xv t) \e^{- \eta_i(\xv) t} \to 
\prod_{j=1}^m \left( \sum_{k\colon\mu^k =\mu_j^{\rm{max}}} v^k_j \right)^{H_{ji}} > 0 
\text{ as } t \to \infty ,
\end{equation}
where
\begin{equation} \label{eq:eta}
\eta_i (\xv) := \sum_{j=1}^m \mu^{\rm{max}}_j (\xv) \, H_{ji} .
\end{equation}

We use this limit in the proof of Lemma~\ref{lem:proper_rays},
which states that $\map$ is proper if and only if it is ``proper along rays''.
For a family of exponential maps,
which are structurally simpler than the composite (\mono-\expo) map~$\map$ in this work,
a corresponding result has been obtained in previous work~\cite[Lemma~11]{MuellerHofbauerRegensburger2019}.

\begin{lemma} \label{lem:proper_rays}
The map~$\map$ defined in Equation~\eqref{eq:F} is proper
if and only if
the ray condition
\[ \tag{$\ast$} \label{eq:ray_condition}
\map(\xv t) \to z \text{ as }t \to \infty \quad \text{implies} \quad z\in\partial\R^d_\ge 
\]
holds for all nonzero $\xv\in\R^{\dP}$. 
\end{lemma}
\begin{proof}
We assume that the ray condition \eqref{eq:ray_condition}
holds. 

In order to apply Lemma~\ref{lem:proper},
we consider a sequence on the unit sphere, that is, 
$\xv_n$ in $\R^{\dP}$ with $|\xv_n|=1$ and $\xv_n \to \xv$ (with $|\xv|=1$),
and a sequence $t_n$ in $\R_>$ with $t_n \to \infty$.

First, we show
\[
|\map(\xv t)|\to \infty \text{ as } t\to \infty \quad \text{implies} \quad |\map(\xv_nt_n)|\to \infty \text{ as } n\to \infty .
\]

Clearly, $|\map(\xv t)|\to \infty$ as $t\to \infty$ implies that there exists an index $i\in\{1,\ldots,d\}$ such that
\[
\map_i(\xv t) \to \infty \text{ as } t\to \infty .
\]

Using the limit~\eqref{eq:limit}, this implies
\[
\eta_i(\xv) > 0 ,
\]
cf.~Definition~\eqref{eq:eta}.

Next, we consider $\xv'$ close to $\xv$,
and introduce $\mu'^k := \mu^k(\xv') = w^k \cdot \xv'$, 
\[ \mu'^{\max}_j := \mu^{\max}_j(\xv') = \max_{k\colon v_j^k\neq 0} \, \mu^k (\xv') , \] 
and 
\[ \eta_i' := \eta_i(\xv') = \sum_{j=1}^m \mu^{\rm{max}}_j (\xv') \, H_{ji} \] 
close to $\eta_i$. 
In particular, $\eta_i'>\frac{\eta_i}{2}$.
Using the limit~\eqref{eq:limit} again (for $\xv'$ rather than $\xv$),
that is,
\[
\map_i(\xv't) \e^{- \eta'_i t} \to 
\prod_{j=1}^m \left( \sum_{k\colon\mu'^k =\mu'^{\rm{max}}_j}  \, v^k_j \right)^{H_{ji}} > 0
\text{ as } t \to \infty ,
\]
%
there is $\gamma(\xv')>0$
such that
\[
\map_i(\xv't) \e^{- \eta'_i t} > 
\gamma(\xv') 
\text{ as } t \to \infty .
\]

Since this holds for any $\xv'$,
there is $\gamma > 0$ (independent of $\xv'$)
such that
\[
\map_i(\xv't) \e^{- \eta'_i t} > \gamma 
\text{ as } t \to \infty ,
\]
and since $\eta_i'>\frac{\eta_i}{2}$,
\[
\map_i(\xv't) \e^{- \frac{\eta_i}{2} t} > \gamma 
\text{ as } t \to \infty .
\]

Now, for the sequences $\xv_n$ and $t_n$ (in Lemma~\ref{lem:proper}),
\[ 
\map_i(\xv_n t_n) \e^{- \frac{\eta_i}{2} t_n} > \gamma  
\text{ as } n \to \infty 
\]
and hence
\[ 
\map_i(\xv_n t_n) \to \infty \text{ as } n \to \infty 
\]
such that
\[ 
|\map(\xv_nt_n)| \to \infty \text{ as } n\to\infty ,
\]
as claimed.

Finally,
let $\map(\xv_nt_n)\to z'$ as $n\to \infty$. 
By the argument above,
$\map(\xv t)\to z$ as $t\to\infty$, 
and by the ray condition~\eqref{eq:ray_condition}, 
$z\in \partial \R^d_\ge$. 
In particular, there exists an index $i \in \{1,\ldots, d\}$ such that $\map_i(\xv t) \to 0$ as $t\to \infty$
and hence $\eta_i < 0$, cf.~Equation~\eqref{eq:limit}.
Proceeding as above,
we obtain $\map_i (\xv_n t_n) \to 0$ as $n \to \infty$
and hence $z' \in \partial \R^d_\ge$, as claimed. By Lemma~\ref{lem:proper}, the map~$\map$ is proper.

For the other direction, assume that $\map$ is proper. 
The conditions in Lemma~\ref{lem:proper} hold for all sequences,
in particular, for sequences on rays.
Hence, the ray condition~\eqref{eq:ray_condition} holds.
\end{proof}

Now, we can provide a first characterization of properness.
\begin{theorem} \label{thm:proper_1}
The map~$\map$ defined in Equation~\eqref{eq:F} is proper
if and only if,
for all nonzero $\xv \in \R^{\dP}$, 
there exists $i\in\{1,2,...,d\}$ such that 
\[
\eta_i(\xv) = \sum_{j=1}^m \mu^{\rm{max}}_j(\xv) \, H_{ji} \neq 0 ,
\]
where
\[
\mu_j^{\rm{max}}(\xv) = \max_{k\colon v_j^k\neq 0} \, \basT^\trans v^k \cdot \xv , 
\quad \text{for } j = 1 \ldots, m .
\]
\end{theorem}
\begin{proof}
Using Lemma~\ref{lem:proper_rays},
we characterize when $\map$ is ``proper along rays''.
For nonzero $\xv \in \R^{\dP}$, 
the ray condition~\eqref{eq:ray_condition} is equivalent to either $|\map(\xv t)| \to \infty$ or $\map(\xv t) \to z \in \partial \R^d_\ge$ as $t \to \infty$.
Using the limit~\eqref{eq:limit},
this is equivalent to either $\eta_i(\xv)>0$ for some $i \in \{1,\ldots,d\}$
or $\eta_i(\xv)\le0$ for all $i$ and $\eta_i(\xv)<0$ for some $i$.
It is easy to see that this is equivalent to $\eta_i(\xv)\neq0$ for some $i$.
\end{proof}

In fact, we can characterize properness without the basis matrices $\basT$ and $H$ (of the linear subspaces $T$ and $D$),
but just in terms of the (vertices of the) polytope and the subspaces.

\begin{theorem} \label{thm:proper_2}
The map~$\map$ defined in Equation~\eqref{eq:F} is {\em not} proper
if and only if
there is a nonzero $t \in T$ such that
\[
\mu^{\rm{max}}(t) \in D^\perp ,
\]
where
\[
\mu_j^{\rm{max}}(t) = \max_{k\colon v_j^k\neq 0} \, v^k \cdot t , 
\quad \text{for } j = 1 \ldots, m .
\]
\end{theorem}
\begin{proof}
By Theorem~\ref{thm:proper_1},
the non-properness of $\map$ is equivalent to the existence of a nonzero $\xv \in \R^{\dP}$ such that
$\sum_{j=1}^m \mu^{\rm{max}}_j(\xv) \, H_{ji} = 0$,
for all $i\in\{1,2,...,d\}$,
that is,
\[
H^\trans \mu^{\rm{max}}(\xv) = 0 .
\]
Since
\[
\mu_j^{\rm{max}}(\xv) = \max_{k\colon v_j^k\neq 0} \, \basT^\trans v^k \cdot \xv
= \max_{k\colon v_j^k\neq 0} \, v^k \cdot \basT \xv ,
\]
$\im \basT = T$, and $\ker H^\trans = D^\perp$,
this is further equivalent to the existence of a nonzero $t = \basT \xv \in T$ such that $\mu^{\rm{max}}(t) \in D^\perp$.
\end{proof}


\subsection{\texorpdfstring{$\map$}{pdf} being a diffeomorphism} \label{subsec:diffeo}

We can now combine our results on local invertibilty and properness.

\begin{theorem} \label{thm:main}
The map~$\map$ defined in Equation~\eqref{eq:F}
is a \green{diffeomorphism} 
if and only if
\begin{enumerate}[(i)]
\item 
there do not exist $y \in P$ and nonzero $t \in T$ and $\bar d \in D^\perp$ such that $t = y \circ \bar d$,
and
\item  
for every nonzero $t \in T$,  
\[
\mu^{\rm{max}}(t) \notin D^\perp ,
\]
where
\[
\mu_j^{\rm{max}}(t) = \max_{k\colon v_j^k\neq 0} \, v^k \cdot t , 
\quad \text{for } j = 1 \ldots, m .
\]

\end{enumerate}
\end{theorem}

Recall that, by Propositions~\ref{pro:reform} and~\ref{pro:diff}
and Theorem~\ref{thm:hadamard_maps},
conditions (i) and (ii) in Theorem~\ref{thm:main} are equivalent to $|Y_c|=1$ and $y \in Y_c$ is \green{nondegenerate} for all~$c$,
that is, the existence of a unique, \green{nondegenerate} solution in the coefficient polytope for all parameters.
If $\dim L = n$,
this is further equivalent to $|Z_c|=1$ and $x \in Z_c$ is \green{nondegenerate} for all~$c$,
that is, the existence of a unique, \green{nondegenerate} solution to Equation~\eqref{eq:system} for all parameters.


\subsection{Sufficient conditions} \label{subsec:suf_con}

From Theorem~\ref{thm:local_invert_2},
we immediately obtain a sufficient condition for local invertibility
which only involves sign vectors of linear subspaces.

\begin{proposition} \label{pro:local_invert_suff}
The map~$\map$ defined in Equation~\eqref{eq:F} is locally invertible
if $\sign(T) \cap \sign(D^\perp) = \{0\}$.
\end{proposition}

Such a sign vector condition 
(however, for different subspaces)
appeared in previous work~\cite{MuellerRegensburger2012, Mueller2016}
as an equivalent condition for the injectivity of parametrized families of generalized polynomial/exponential maps.
In the context of reaction networks,
such conditions characterize the nonexistence of multiple complex-balanced equilibria/toric steady states~\cite{PerezMillanEtAl2012,MuellerRegensburger2012,Mueller2016}.

A sufficient condition for properness requires more preparation.
We recall notions from polyhedral geometry
and adapt general definitions for the polytope $\overline P$ (the closure of the coefficient polytope).

For a general polytope $P \subset \R^n$, we denote the face lattice (the set of all faces, partially ordered by set inclusion) by ${\mathcal L}_P$,
and for a set $X \subset \R^n$, we introduce the set of identically zero components,
$\zero(X) = \{ i \in [n] \mid {x_i=0} \text{ for all } {x \in X} \}$.

Recall that $\overline P \subset \R^m_\ge$ lies in an affine subspace with associated linear subspace~${T \subset \R^m}$.
To decide properness via Theorem~\ref{thm:proper_2},
we need to vary over $t \in T$
and determine
\[
\mu_j^{\rm{max}}(t) = \max_{k\colon v_j^k\neq 0} \, v^k \cdot t , 
\quad \text{for } j = 1 \ldots, m .
\]
To guarantee properness,
we partition $T$ into polyhedral cones corresponding to the faces of $\overline P$.
For $\face \in \mathcal{L}_{\overline P}$,
we define the cone
\[
C_\face = \{ t \in T \mid \face = \argmax_{x \in P} \, t \cdot x \} .
\]
In terms of polyhedral geometry,
$C_\face$ is the projection (to $T$) of the relative interior of the normal cone of $\face$.
The cones are disjoint
and ${\bigcup_{\face \in \mathcal{L}_{\overline P}} C_\face = T}$.
Further, for $\face \in \mathcal{L}_{\overline P}$,
we define the sign vector $\tau_\face \in \{0,+\}^m$ with 
\[
\supp(\tau_\face) = \zero(\face) ,
\]
that is, $\tau_\face$ indicates the active nonnegativity constraints that define~$\face$,
explicitly, $(\tau_\face)_i = +$ if and only if $x_i = 0$ for all $x \in \face$.

Now,
let $\mu^* (t) := \max_j \mu^{\rm{max}}_j (t) \in \R$
and
\begin{equation}
\delta\mu^{\rm{max}} (t) := \mu^{\rm{max}} (t) - \mu^* (t) \, \ones_m \in \R^m_\le ,
\end{equation}
that is, $\delta\mu^{\rm{max}} (t)$ is nonpositive.
Since $\ones_m \in D^\perp$,
\[
\delta\mu^{\rm{max}} (t) \in D^\perp
\quad \text{if and only if} \quad
\mu^{\rm{max}} (t) \in D^\perp .
\]

Most importantly,
the signs of $\delta\mu^{\rm{max}} (t)$ are fixed for $t$ within the ``projected normal cones'' introduced above.
\begin{lemma} \label{lem:sign_mu}
Let $P$ be the coefficient polytope. 
Further, let $\face \in \mathcal{L}_{\overline P}$ and $t \in C_\face$.
Then,
\begin{enumerate}[(i)]
\item 
$\delta\mu^{\rm{max}}_j (t)<0$, if $j \in \zero(\face)$.
\item 
$\delta\mu^{\rm{max}}_j (t)=0$, otherwise.
\end{enumerate}
\end{lemma}
\begin{proof}
Let $\ver(\face)$ be the set of vertices of the face~$\face$. 
By the definition of~$C_\face$, 
\begin{itemize}
    \item $v^k \cdot t = \mu^* (t)$ for $v^k \in \ver(\face)$.
    \item $v^k \cdot t < \mu^* (t)$ for $v^k \not\in \ver(\face)$.
\end{itemize}
We consider the two cases.
\begin{enumerate}[(i)]
\item 
$j \in\zero(\face)$: 
Here, 
$v^k \in \ver(\face)$ implies $v^k_j = 0$,
that is,
$v^k_j \neq 0$ implies $v^k \not\in \ver(\face)$.
Hence,
\[
\mu_j^{\rm{max}} (t) 
= \max_{k \colon v_j^k\neq 0} \, v^k \cdot t
< \mu^* (t)
\]
and $\delta\mu_j^{\max}(t) = \mu_j^{\rm{max}} (t) - \mu^* (t) < 0$.
\item 
$j \not\in \zero(\face)$: 
Here, there is $v^k \in \ver(\face)$ such that $v^k_j \neq 0$.
Hence,
\[
\mu_j^{\rm{max}} (t) 
= \max_{k \colon v_j^k\neq 0} \, v^k \cdot t
= \mu^* (t)
\]
and $\delta\mu_j^{\max}(t) = \mu_j^{\rm{max}} (t) - \mu^* (t) = 0$.
\end{enumerate}
\end{proof}

Now, we can provide a sufficient condition for properness
involving only sign vectors of the linear subspace $D$ and the polytope $P$.
\begin{proposition} \label{pro:proper_suff}
If the map~$\map$ defined in Equation~\eqref{eq:F} is {\em not} proper,
then there exists a face $\face \in \mathcal{L}_{\overline P}$
with corresponding sign vector $\tau_\face \in \{0,+\}^m$ 
such that $\tau_\face \in \sign(D^\perp)$.
\end{proposition}
\begin{proof}
By Theorem~\ref{thm:proper_2},
if the map~$\map$ is not proper, there exists a nonzero $t \in T$ such that 
\[
- \delta\mu^{\rm{max}} (t) \in D^\perp .
\]
Let $\face \in \mathcal{L}_{\overline P}$ such that $t \in C_\face$.
By Lemma~\ref{lem:sign_mu}, $-\delta\mu^{\rm{max}}_j (t) > 0$ if $j \in \zero(\face)$,
and $\delta\mu^{\rm{max}}_j (t) = 0$ otherwise. 
That is, $\sign(-\delta\mu^{\max} (t)) = \tau_\face$,
and hence $\tau_\face \in \sign(D^\perp)$.
\end{proof}

Altogether, we obtain sufficient conditions for the map~$\map$ being a diffeomorphism.

\begin{theorem} \label{thm:suff}
The map~$\map$ defined in Equation~\eqref{eq:F} is a \green{diffeomorphism} if
\begin{enumerate}[(i)]
\item
$\sign(D^{\perp}) \cap \sign(T) = \{0\}$, and
\item 
for every face $\face \in \mathcal{L}_{\overline P}$ 
with corresponding sign vector $\tau_\face \in \{0,+\}^m$, 
$\tau_\face \not\in \sign(D^\perp)$.
\end{enumerate}
\end{theorem}
\begin{proof}
By Propositions~\ref{pro:local_invert_suff} and \ref{pro:proper_suff} 
(based on Theorem~\ref{thm:hadamard_maps} and Proposition~\ref{pro:diff}).
\end{proof}

Theorem~\ref{thm:suff} only involves sign vectors of linear subspaces and of the face lattice of the coefficient polytope. Hence, it can appropriately be referred to as a multivariate Descartes' rule of signs for exactly one solution.


\section{Example} \label{sec:example}

We provide an example to illustrate the geometric objects introduced in this work and to demonstrate the application of our results.
As noted in the introduction, 
equivalent conditions for the inhomogeneous problem (iii)
are trivially sufficient for the homogeneous problem (III), but are too strong in general.
Whereas the results from previous work~\cite{MuellerHofbauerRegensburger2019} are not applicable in the example,
our sufficient (sign vector) conditions for the existence of a unique, nondegenerate solution do hold.
 
\newcommand{\ap}{\,+\,}
\newcommand{\am}{\,-\,}

\begin{example} 
The parametrized system
\begin{alignat}{9}
c_1 \, x^2y^{-1} &\ap& c_2 \, xy &\am& 2 \, c_3 \, x^{-1}y && &\am& c_5 \, y &= 0 , \\
- c_1 \, x^2y^{-1} &\ap& 3 \, c_2 \, xy &\ap& c_3 \, x^{-1}y &\ap& c_4 \, x &\ap& 2 \, c_5 \, y &= 0 \nonumber
\end{alignat}
can be written as $A \, (c\circ x^B)=0$, where
\[
A=
\begin{pmatrix}
1 & 1 & -2 & 0 & -1 \\
-1 & 3 & 1 & 1 & 2
\end{pmatrix} ,
\]
\[
B=
\begin{pmatrix}
2 & 1 & -1 & 1 & 0\\
-1 & 1 & 1 & 0 & 1
\end{pmatrix} ,
\]
and 
\[
c = (c_1,c_2,c_3,c_4,c_5)^\trans .
\]
Hence,
\[
\mathcal{A}=
\begin{pmatrix}
1 & 1 & -2 & 0 & -1 \\
-1 & 3 & 1 & 1 & 2 \\
1 & 1 & 1 & 1 & 1 
\end{pmatrix} ,
\]
\[
\mathcal{B}=
\begin{pmatrix}
2 & 1 & -1 & 1 & 0 \\
-1 & 1 & 1 & 0 & 1 \\
1 & 1 & 1 & 1 & 1 
\end{pmatrix} ,
\]
$T = \ker \mathcal{A} = \im \basT$ with
\[
\basT = \begin{pmatrix}
1  & 1\\
1  & -5\\
1  & -8\\
-3 & 0\\
0  & 12
\end{pmatrix} ,
\]
and
$D = \ker \mathcal{B} = \im H$ with
\[
H = \begin{pmatrix}
1  & 0\\
0  & 1\\
0  & 1\\
-2 & 0\\
1  & -2
\end{pmatrix} .
\]
In particular, $\dP = \dim(T) = 2$ and $d = \dim(D) = 2$.

Further,
the polytope $\overline P$ has 3 vertices, 
\[
v^1 = 
\frac{1}{5}
\begin{pmatrix}
3\\
0\\
1\\
0\\
1
\end{pmatrix}, 
\quad
v^2 = 
\frac{1}{4}
\begin{pmatrix}
2\\
0\\
1\\
1\\
0
\end{pmatrix}, 
\quad
v^3 = 
\frac{1}{12}
\begin{pmatrix}
7\\
1\\
4\\
0\\
0
\end{pmatrix} ,
\]
which can be obtained as the normalized, nonnegative elementary vectors of $\ker A$,
using the {\tt SageMath} package~\cite{package_elementary_vectors}. See also~\cite{MuellerRegensburger2016}.

The map~$\map$ in the simplified form of Equation~\eqref{eq:F_simple} is given by
\[
\map \colon \R^2 \to \R^2_>, \quad
\xx \mapsto \left( \sum_{k=1}^{3} \e^{ v^k \cdot \basT \xx } v^k \right)^{H} .
\]
Explicitly,
\[
\map(\xx) = \left(v(\xx)\right)^H
= \begin{pmatrix} v_1 \, (v_4)^{-2} \, v_5 \\ v_2 \, v_3 \, (v_5)^{-2} \end{pmatrix}
\]
with
\begin{align*}
v = v(\xx) = \e^{\frac{1}{5} (4,7) \xx} v^1 + \e^{\frac{1}{4} (0,-6) \xx} v^2 + \e^{\frac{1}{12} (12,-30) \xx} v^3 \in \R^5_> .
\end{align*}

To guarantee bijectivity,
we first verify $\sign(D^{\perp})\cap \sign(T) = \{0\}$, using the {\tt SageMath} package~\cite{package_sign_vector_conditions}. 
See also~\cite{AichmayrMuellerRegensburger2024}.
By Proposition~\ref{pro:local_invert_suff}, 
the map~$\map$ is locally invertible.

Second, we verify that,
for all faces $\face \in \mathcal{L}_{\overline P}$,
$\tau_\face \notin \sign(D^\perp)$.
For the vertex~$v^1$,
\[
\tau_{v^1} = (0,+,0,+,0)^\trans ,
\]
and there is no vector $u \in \R^5$ with $\sign(u) = \tau_{v^1}$ in $\ker H^\trans = D^\perp$.
The same holds for $v^2$ and $v^3$. 
Further, for the edge $e^{12} = \overline{v^1v^2}$,
\[
\tau_{e^{12}} = (0,+,0,0,0)^\trans ,
\]
and again there is no vector $u \in \R^5$ with $\sign(u) = \tau_{v^1}$ in $\ker H^\trans = D^\perp$.
The same holds for the edges $e^{23}$ and $e^{31}$.
By Proposition~\ref{pro:proper_suff}, 
the map~$\map$ is proper.

By Theorem~\ref{thm:suff}, the map~$\map$ is a \green{diffeomorphism},
and by Propositions~\ref{pro:reform} and \ref{pro:diff},
$|Y_c|=1$ and $y \in Y_c$ is nondegenerate for all~$c$. 

However, \cite[Theorem 14, condition (ii)]{MuellerHofbauerRegensburger2019} (with~$S$ and~$\tilde S$ substituted by~$T$ and~$D$) does not hold for the example.
\end{example}


\section{Discussion} \label{sec:discussion}

In this work, we have characterized the existence of a unique, \green{nondegenerate} solution $x$
(modulo an exponential manifold) to the homogeneous, parametrized system ${A \, (c \circ x^B) = 0}$ for all parameters $c$,
whereas
previous work~\cite{CraciunGarcia-PuenteSottile2010,MuellerRegensburger2012,Mueller2016,MuellerHofbauerRegensburger2019} concerned uniqueness, existence, or unique existence for the inhomogeneous system ${A \, (c \circ x^B) = z}$. 
Trivially, unique existence for all (suitable) right-hand sides~$z$ and all parameters~$c$ {\em implies} unique existence for $z=0$.
Hence, the (equivalent or sufficient) conditions obtained for the inhomogeneous system in previous work~\cite{MuellerHofbauerRegensburger2019} {\em imply} the equivalent conditions for the homogeneous system obtained in this work, but are too strong in general.
Indeed, as demonstrated by the example in Section~\ref{sec:example},
the results from previous work may not be applicable,
whereas the equivalent conditions in Theorem~\ref{thm:main}
(and even the sufficient conditions in Theorem~\ref{thm:suff})
may hold.

Finally, we mention computational aspects.
The condition for local invertibility is a quadratic quantifier elimination problem~\cite{weispfenning1997quantifier},
and the condition for properness can be translated into a finite number of linear programs.
In a future paper, 
we plan to provide efficient implementations.
In particular,
we plan to extend the {\tt SageMath} packages \cite{package_elementary_vectors,package_sign_vector_conditions},
which already cover the conditions in~\cite{MuellerHofbauerRegensburger2019} as described in~\cite{AichmayrMuellerRegensburger2024}.


\subsubsection*{Acknowledgements.} 
SM was supported by the Austrian Science Fund (FWF), grant DOIs: 
10.55776/P33218 and 10.55776/PAT3748324. 



\bibliographystyle{abbrv}

\bibliography{old,new,MR}

@misc{package_sign_vector_conditions,
  author   = {Aichmayr, Marcus},
  title    ={\textnormal{\texttt{sign\_vector\_conditions}}},
  note    = "\url{{https://github.com/MarcusAichmayr/sign\_vector\_conditions}}",
  version  = {1.0},
}

@misc{package_elementary_vectors,
  author   = {Aichmayr, Marcus},
  title    = {\textnormal{\texttt{elementary\_vectors}}},
  note    = "\url{{https://github.com/MarcusAichmayr/elementary\_vectors}}",
  version  = {1.0},
}

@InProceedings{AichmayrMuellerRegensburger2024,
author="Aichmayr, Marcus S.
and M{\"u}ller, Stefan
and Regensburger, Georg",
title="A {S}age{M}ath Package for Elementary and Sign Vectors with Applications to Chemical Reaction Networks",
booktitle="Mathematical Software -- ICMS 2024",
series="Lecture Notes in Compt. Sci.",
volume="14749",
year="2024",
publisher="Springer Nature Switzerland",
address="Cham",
pages="155--164",
}

@article{MuellerRegensburger2023a,
    author = {M{\"u}ller, Stefan and Regensburger, Georg},
    title = {Parametrized systems of generalized polynomial
inequalities via linear algebra and convex geometry},
    year = {2026},
    journal = {Positivity},
    volume = {30},
    pages = {4},
    doi = {10.1007/s11117-025-01158-4}
}

@article{MuellerRegensburger2023b,
    author = {M{\"u}ller, Stefan and Regensburger, Georg},
    title = {\phantom{}{P}arametrized systems of generalized polynomial equations: first applications to fewnomials},
    year = {2023},
    journal = {arXiv},
    note = {\href{https://arxiv.org/abs/2304.05273}{arXiv:2304.05273} [math.AG]},
}

@Article{MuellerHofbauerRegensburger2019,
  author  = {M{\"u}ller, Stefan and Hofbauer, Josef and Regensburger, Georg},
  title   = {On the bijectivity of families of exponential/generalized polynomial maps},
  journal = {SIAM J. Appl. Algebra Geom.},
  year    = {2019},
  Volume = {3},
  number = {3},
  Pages   = {412--438},
  doi = {10.1137/18M1178153}
}

@Article{FeliuMuellerRegensburger2019,
author = {Feliu, Elisenda and M{\"u}ller, Stefan and Regensburger, Georg},
title = {Characterizing injectivity of classes of maps via classes of matrices},
journal = {Linear Algebra and its Applications},
year = {2019},
Volume  = {580},
Pages    = {236--261},
doi = {10.1016/j.laa.2019.06.015}
}

@Article{Mueller2016,
  Title       = {Sign conditions for injectivity of generalized polynomial maps with applications to chemical reaction networks and real algebraic geometry},
  Author   = {M{\"u}ller, Stefan and Feliu, Elisenda and Regensburger, Georg and Conradi, Carsten and Shiu, Anne and Dickenstein, Alicia},
  Journal  = {Found. Comput. Math.},
  Fjournal = {Foundations of Computational Mathematics. The Journal of the Society for the Foundations of Computational Mathematics},
  Volume = {16},
  number = {1},
  Pages   = {69--97},
  Year      = {2016},
  doi = {10.1007/s10208-014-9239-3}
}

@InProceedings{MuellerRegensburger2014,
  Author                   = {M{\"u}ller, Stefan and Regensburger, Georg},
  Title                    = {Generalized Mass-Action Systems and Positive Solutions of Polynomial Equations with Real and Symbolic Exponents},
  Booktitle                = {Computer Algebra in Scientific Computing -- CASC 2014},
  Year                     = {2014},
  Pages                    = {302--323},
  Series                   = {Lecture Notes in Comput. Sci.},
  Volume                   = {8660},
  Publisher                = {Springer},
  Address                  = {Berlin/Heidelberg},
  doi = {10.1007/978-3-319-10515-4_22}
}

@Article{MuellerRegensburger2012,
  Author                   = {M{\"u}ller, Stefan and Regensburger, Georg},
  Title                    = {Generalized mass action systems: {C}omplex balancing equilibria and sign vectors of the stoichiometric and kinetic-order subspaces},
  Journal                  = {SIAM J. Appl. Math.},
  Year                     = {2012},
  Volume                   = {72},
  number                    = {6},
  Pages                    = {1926--1947},
  doi = {10.1137/110847056}
}

@Article{MuellerRegensburger2016,
  Author                   = {M{\"u}ller, Stefan and Regensburger, Georg},
  Title                    = {Elementary vectors and conformal sums in polyhedral geometry and their relevance for metabolic pathway analysis},
  Journal                  = {Front. Genet.},
  Year                       = {2016},
  Volume                  = {7},
  Number                 = {90},
  Pages                    = {1--11},
  doi = {10.3389/fgene.2016.00090}
}

@misc{ChecaFeliu2025,
      title={An effective criterion for multiple positive zeros of vertically parametrized polynomial systems}, 
      author={Carles Checa and Elisenda Feliu},
      year={2025},
      eprint={2512.07560},
      archivePrefix={arXiv},
      primaryClass={math.AG},
      url={https://arxiv.org/abs/2512.07560}, 
}

@misc{BF2026,
      title={Positive equilibria in mass action networks: geometry and bounds}, 
      author={Murad Banaji and Elisenda Feliu},
      year={2026},
      eprint={2409.06877},
      archivePrefix={arXiv},
      primaryClass={q-bio.MN},
      url={https://arxiv.org/abs/2409.06877}, 
}

@article{weispfenning1997quantifier,
  title={Quantifier elimination for real algebra—the quadratic case and beyond},
  author={Weispfenning, V.},
  journal={Applicable Algebra in Engineering, Communication and Computing},
  volume={8},
  pages={85--101},
  year={1997},
  publisher={Springer}
}

@article{banach1934mehrdeutige,
  title={{\"U}ber mehrdeutige stetige {A}bbildungen},
  author={Banach, S. and Mazur, S.},
  journal={Stud. Math},
  volume={5},
  number={1},
  pages={174--178},
  year={1934},
  publisher={Polska Akademia Nauk. Instytut Matematyczny PAN}
}

@article{gordon1972diffeomorphisms,
  title={On the diffeomorphisms of Euclidean space},
  author={Gordon, W.},
  journal={Am. Math. Mon.},
  volume={79},
  number={7},
  pages={755--759},
  year={1972},
  publisher={Taylor \& Francis}
}

@article{bihan2017descartes,
  title={Descartes’ rule of signs for polynomial systems supported on circuits},
  author={Bihan, F. and Dickenstein, A.},
  journal={Int. Math. Res. Not.},
  volume={2017},
  number={22},
  pages={6867--6893},
  year={2017},
  publisher={Oxford University Press}
}

@article{bihan2020sign,
  title={Sign conditions for the existence of at least one positive solution of a sparse polynomial system},
  author={Bihan, F. and Dickenstein, A. and Giaroli, M.},
  journal={Adv. Math.},
  volume={375},
  pages={107412},
  year={2020},
  publisher={Elsevier}
}

@article{bihan2020lower,
  title={Lower bounds for positive roots and regions of multistationarity in chemical reaction networks},
  author={Bihan, F. and Dickenstein, A. and Giaroli, M.},
  journal={J. Alg.},
  volume={542},
  pages={367--411},
  year={2020},
  publisher={Elsevier}
}

@article{bihan2021optimal,
  title={Optimal {D}escartes’ rule of signs for systems supported on circuits},
  author={Bihan, F. and Dickenstein, A. and Forsg{\aa}rd, J.},
  journal={Math. Ann.},
  volume={381},
  number={3},
  pages={1283--1307},
  year={2021},
  publisher={Springer}
}

@article{feliu2022generalizing,
  title={On generalizing {D}escartes' rule of signs to hypersurfaces},
  author={Feliu, E. and Telek, M.},
  journal={Adv. Math.},
  volume={408},
  pages={108582},
  year={2022},
  publisher={Elsevier}
}

@inproceedings{khovanskii1980class,
  title={A class of systems of transcendental equations},
  author={Khovanskii, A.},
  booktitle={Doklady Akademii Nauk},
  volume={255},
  pages={804--807},
  year={1980},
  organization={Russian Academy of Sciences}
}

@book{Khovanskii1991,
  title={Fewnomials},
  author={Khovanskii, A.},
  volume={88},
  year={1991},
  publisher={American Mathematical Soc.}
}

@book{sottile2011real,
  title={Real solutions to equations from geometry},
  author={Sottile, F.},
  volume={57},
  year={2011},
  publisher={American Mathematical Soc.}
}

@article{telek2023topological,
  title={Topological descriptors of the parameter region of multistationarity: Deciding upon connectivity},
  author={Telek, M{\'a}t{\'e} L{\'a}szl{\'o} and Feliu, Elisenda},
  journal={PLOS Comp. Biol.},
  volume={19},
  number={3},
  pages={e1010970},
  year={2023},
  publisher={Public Library of Science San Francisco, CA USA}
}

@article{bihan2007new,
  title={New fewnomial upper bounds from {G}ale dual polynomial systems, 2007},
  author={Bihan, F. and Sottile, F.},
  journal={Mosc. Math. J.},
  volume={7},
  number={3},
  year={2007}
}

@incollection{bihan2008sharpness,
  title={On the sharpness of fewnomial bounds and the number of components of fewnomial hypersurfaces},
  author={Bihan, F. and Rojas, M. and Sottile, F.},
  booktitle={Algorithms in algebraic geometry},
  pages={15--20},
  year={2008},
  publisher={Springer}
}

@article {BanajiPantea2016,
    AUTHOR = {Banaji, Murad and Pantea, Casian},
     TITLE = {Some results on injectivity and multistationarity in chemical
              reaction networks},
   JOURNAL = {SIAM J. Appl. Dyn. Syst.},
  FJOURNAL = {SIAM Journal on Applied Dynamical Systems},
    VOLUME = {15},
      YEAR = {2016},
    NUMBER = {2},
     PAGES = {807--869},
      ISSN = {1536-0040},
   MRCLASS = {80A30 (15A15 37N99 92C42)},
  MRNUMBER = {3490491},
MRREVIEWER = {Marc R. Roussel},
       DOI = {10.1137/15M1034441},
       URL = {https://doi.org/10.1137/15M1034441},
}

@InProceedings{CraciunGarcia-PuenteSottile2010,
  Title                    = {Some geometrical aspects of control points for toric patches},
  Author                   = {Craciun, G. and Garcia-Puente, L. and Sottile, F.},
  Booktitle                = {Mathematical Methods for Curves and Surfaces},
  Year                     = {2010},
  Editor                   = {D\ae{}hlen, M and Floater, M S and Lyche, T and Merrien, J-L and Morken, K and Schumaker, L L},
  Pages                    = {111--135},
  Publisher                = {Springer},
  Series                   = {Lecture Notes in Comput. Sci.},
  Volume                   = {5862},

  Optaddress               = {Heidelberg},
  Optnote                  = {MMCS 2008, T\o{}nsberg, Norway, June 26-July 1, 2008, Revised Selected Papers}
}

@Article{Feinberg1972,
  Title                    = {Complex balancing in general kinetic systems},
  Author                   = {Feinberg, M.},
  Journal                  = {Arch. Rational Mech. Anal.},
  Year                     = {1972},
  Pages                    = {187--194},
  Volume                   = {49},

  Fjournal                 = {Archive for Rational Mechanics and Analysis},
  ISSN                     = {0003-9527},
  Mrclass                  = {82.35},
  Mrnumber                 = {0413930},
  Mrreviewer               = {Ole J. Heilmann},
  doi = {10.1007/BF00255665}
}

@Book{Fulton1993,
  Title                    = {Introduction to toric varieties},
  Author                   = {Fulton, W.},
  Publisher                = {Princeton University Press},
  Year                     = {1993},

  Address                  = {Princeton, NJ},
  Series                   = {Ann. of Math. Stud.},
  Volume                   = {131},

  ISBN                     = {0-691-00049-2},
  Mrclass                  = {14M25 (14-02 14J30)},
  Mrnumber                 = {1234037},
  Mrreviewer               = {T. Oda},
  Pages                    = {xii+157},
  doi = {10.1007/0-387-27103-1_10}
}

@Article{Horn1972,
  Title                    = {Necessary and sufficient conditions for complex balancing in chemical kinetics},
  Author                   = {Horn, F.},
  Journal                  = {Arch. Rational Mech. Anal.},
  Year                     = {1972},
  Pages                    = {172--186},
  Volume                   = {49},

  Fjournal                 = {Archive for Rational Mechanics and Analysis},
  ISSN                     = {0003-9527},
  Mrclass                  = {82.35},
  Mrnumber                 = {0413929},
  Mrreviewer               = {Ole J. Heilmann},
  doi = {10.1007/BF00255664}
}

@Article{HornJackson1972,
  Title                    = {General mass action kinetics},
  Author                   = {Horn, F. and Jackson, R.},
  Journal                  = {Arch. Rational Mech. Anal.},
  Year                     = {1972},
  Pages                    = {81--116},
  Volume                   = {47},

  Fjournal                 = {Archive for Rational Mechanics and Analysis},
  ISSN                     = {0003-9527},
  Mrclass                  = {80.34},
  Mrnumber                 = {0400923},
  Mrreviewer               = {J. Adler},
  doi = {10.1007/BF00251225}
}

@Article{PerezMillanEtAl2012,
  Title                    = {Chemical reaction systems with toric steady states},
  Author                   = {P{\'e}rez Mill{\'a}n, M. and Dickenstein, A. and Shiu, A. and Conradi, C.},
  Journal                  = {Bull. Math. Biol.},
  Year                     = {2012},
  Number                   = {5},
  Pages                    = {1027--1065},
  Volume                   = {74},

  Doi                      = {10.1007/s11538-011-9685-x},
  Fjournal                 = {Bulletin of Mathematical Biology},
  ISSN                     = {0092-8240},
  Mrclass                  = {92E20 (13F20 34D20)},
  Mrnumber                 = {2909119},
  Url                      = {http://dx.doi.org/10.1007/s11538-011-9685-x}
}

\end{document}